\documentclass[12pt]{article} 

\usepackage{amsmath} % Enhanced math environment
\usepackage{amssymb} % Additional math symbols
\usepackage{amsthm} % Theorem environment
\usepackage{newtxtext} % Times font

\usepackage{graphicx} % Enhanced support for graphics

\usepackage{float} % Improved interface for floating objects
\usepackage{caption} % Customizing captions
\usepackage[hidelinks,colorlinks=false]{hyperref} % For hyperlinks within the document

\usepackage{cleveref} % Enhanced cross-referencing
\makeatletter
\renewcommand*{\@fnsymbol}[1]{%
  \ifcase#1%
    \or 1\or 2\or 3\or 4\or 5%
    \or 6\or 7\or 8\or 9\or 10%
    \or 11\or 12\or 13\or 14\or 15%
    \or 16\or 17\or 18\or 19\or 20%
    \or 21\or 22\or 23\or 24\or 25%
    \or 26%
  \else
    \@ctrerr
  \fi
}

\renewcommand*{\@makefnmark}{\@fnsymbol\c@footnote}
\makeatother
\usepackage[a4paper,margin=0.8in]{geometry} % Adjust margins
\usepackage{setspace} % Adjust line spacing
\usepackage{bm} % Bold math symbols
\usepackage{fontawesome}
\usepackage{marvosym}
\usepackage[authoryear, round]{natbib} % Flexible bibliography support

\newtheorem{theorem}{Theorem}[section]
\newtheorem{lemma}[theorem]{Lemma}

\newtheorem{definition}{Definition}[section]
\newtheorem{example}{Example}[section]
\newtheorem{remark}{Remark}[section]
\newtheorem{proposition}[theorem]{Proposition}

\begin{document}
% Title and author

\title{ \textbf{Stochastic comparison of series and parallel systems' lifetime in Archimedean copula under random shock}}
\author{
\textbf{Sarikul Islam}\thanks{Department of Mathematics, Indian Institute of Technology Kharagpur, Kharagpur 721302, India. \newline \text{\quad} Email:  \href{mailto:sarikul_phd_math@kgpian.iitkgp.ac.in}{sarikul\_phd\_math@kgpian.iitkgp.ac.in}} 
\and 
\textbf{Nitin Gupta}\thanks{Department of Mathematics, Indian Institute of Technology Kharagpur, Kharagpur 721302, India.\newline\text{\quad} Email: \href{mailto:nitin.gupta@maths.iitkgp.ac.in}{nitin.gupta@maths.iitkgp.ac.in}}
}

  \date{}               
\maketitle

\section*{Abstract}
In this paper, we studied the stochastic ordering behavior of series as well as parallel systems' lifetimes comprising dependent and heterogeneous components, experiencing random shocks, and exhibiting distinct dependency structures. We establish certain conditions on the lifetime of individual components where the dependency among components defined by Archimedean copulas, and the impact of random shocks on the overall system lifetime to get the results. We consider components whose survival functions are either increasing log-concave or decreasing log-convex functions of the parameters involved. These conditions make it possible to compare the lifetimes of two systems using the usual stochastic order framework. Additionally, we provide examples and graphical representations to elucidate our theoretical findings.

\subsection*{Keywords} Usual stochastic order; majorization; super-additive(Sub-additive) function; schur-convex(Schur-concave) function.\\ \\
\textbf{2020 MSC Classification: } 60E15; 90B25; 62G30.
\section{Introduction} In reliability theory, a  \( n \)-component system is termed as \( k \)-out-of-\( n \) system if it remains operational if at least \( k \) of the total \(n\) components are working. Denote the \( k \)-th order statistic by \( X_{k:n} \), of the random variables \( X_1, \ldots, X_n \) which denotes the lifetime of a \( (n - k + 1) \) out-of-\( n \) system. In the past few decades, some articles have studied the stochastic ordering of order statistics when the random variables $X=(X_1, \cdots, X_n)$ represent the lifetimes of components, assuming that the component's lifetimes are identically and independently distributed. For example, \citet{PROSCHAN1976608}, \citet{pledger1971comparisons}, \citet{bergmann1991stochastic}, \citet{nelsen2006archimedean}, \citet{kochar2007stochastic}, and \citet{kochar2012stochastic} have studied these systems and compared the lifetimes of different systems. However, the lifetimes of components in a system may or may not be independently and identically distributed. To deal with the dependency and heterogeneity among the lifetimes of components, the multivariate Archimedean copula plays a significant role in the stochastic ordering of order statistics and comparing the reliabilities of systems, for reference few articles are \citet{balakrishnan2013ordering}, \citet{li2015ordering}, \citet{balakrishnan2020exponentiated}, \citet{erdely2014frank} and \citet{zhang2019stochastic} for the relevant theory dealing with non-identical and dependent components. In practical scenarios, the components may exhibit dependencies, and different systems may exhibit different dependency structures among their components. Furthermore, various external random shocks in the components may affect the lifetimes of the system. \citet{article} considered the case when the components are independently distributed and derived results for the stochastic ordering of the lifetimes of two fail-safe systems subject to random shock. \citet{balakrishnan2018ordering} derived some results regarding the largest claim amounts from two heterogeneous portfolios. However, the claims are assumed to be independent. Reliability systems often encounter shocks from external stress factors, which occur randomly and can significantly impact the system's reliability. The most recent article by \citet{amini2024comparison} focuses on the stochastic ordering of the lifetimes of series and parallel systems under the assumption that components are dependent and non-identical, all the systems exhibit the same dependency structure, represented by the same Archimedean copula, under the presence of random shocks in the components. But, in general, different systems may exhibit different dependency structures among their components' lifetime. So, It's naturally interesting to explore the circumstance in which the components within the system undergo random shocks, and survive after the random shocks with certain probabilities, and different systems exhibit different dependency structures among the components' lifetime. \citet{das2022ordering} studied the smallest insurance claim amounts from two portfolios using a very restricted model (exponentiated location-scale model) with several conditions imposed on the parameters and the generator. In this document, we'll study the stochastic comparison of lifetimes between series systems as well as parallel systems. We consider that the dependency structures among the components' lifetimes are represented by different Archimedean copulas with generators denoted by $\phi_1$ and $\phi_2$ (with pseudo-inverses $\psi_1$ and $\psi_2$ respectively). The components of the systems may experience random shocks, and their lifetimes follow distribution models with increasing log-concave or decreasing log-convex survival functions in parameters, which are very general models in statistical analyses.\newline

The paper is structured as follows: Section two presents several established concepts and theorems in the form of lemmas that are useful for the remaining sections. Section three discusses some results for the stochastic ordering of series and parallel systems in the presence of random shock when the composition function $\psi_2\circ\phi_1$ is super-additive. Section four delves into the case when the composition function $\psi_2\circ\phi_1$ is sub-additive. Section five discusses some potential applications of the derived results. We conclude the manuscript in section six. 
\section{Preliminaries}
Let's begin by revisiting certain concepts and important lemmas that are fundamental to the subsequent main findings. We employ the following notations where bold font represents vectors in $\mathbb{R}^n$.\newline  \( \mathbb{R} = (-\infty, \infty),\:\mathbb{R}^+ = (0, \infty),\:\mathbb{D}= \{\bm{\alpha}\:: \alpha_1 \geq \ldots \geq \alpha_n\},\:\mathbb{D}^+= \{\bm{\alpha}\:: \alpha_1 \geq \ldots \geq \alpha_n>0\},\, \mathbb{I} = \{\bm{\alpha}\:: \alpha_1 \leq \ldots \leq \alpha_n\},\: \mathbb{I}^+ = \{\bm{\alpha}\:: 0 < \alpha_1 \leq \ldots \leq \alpha_n\}
\)\newline\newline
 Let $X = (X_1, \ldots, X_n)$ be a random vector with dependent and non-identically distributed random variables and $X_i \sim F(\cdot; \alpha_i)$, where $F(\cdot; \alpha_i)$ represents the CDF (cumulative distribution function) of $X_i$ and $\alpha_i(>0)$ denotes the parameter of the distribution of $X_i$, for $i = 1, \ldots, n$. Suppose that the survival function of $X_i$, for $i = 1, \ldots, n$, is denoted by $\overline{F}(x; \alpha_i) = 1 - F(x; \alpha_i)$. Moreover, consider $I_{p_i}\sim Bernoulli(p_i)$ be independent of $X_i$'s, $E[I_{p_i}] = p_i$, for $i = 1, \ldots, n$. Here, $I_{p_1}, \ldots, I_{p_n}$ are employed to signify whether the system's components, having random lifetimes $X_1, \ldots, X_n$, have survived or failed after experiencing the random shock. Specifically, if $I_{p_i} = 1$, then  $i^{th}$ component survives after experiencing a random shock; else, if $I_{p_i} = 0$, the $i^{th}$ component fails to survive after experiencing a random shock, for $i = 1, \ldots, n$. Therefore, the random variables $X_1I_{p_1}, \ldots, X_nI_{p_n}$ represent the lifetimes of the systems' components, subject to random shock effect and we denote $X_{p_i}:= X_iI_{p_i}$, for $i = 1, \ldots, n$. Hence, $X_{p_i}$ is comprising two random variables: first one is $X_i$ with CDF $F_{X_i}(x,\alpha_i)$, the second one is a Bernoulli random variable $I_{p_i}$ with mean $p_i$. Following simple proof steps we find, the survival function of $X_{p_i}:= X_iI_{p_i}$ is $\overline{F}_{X_{p_i}}(x) = p_i \overline{F}_{X_i}(x)$, for each $i = 1, \ldots, n$. Therefore,
\[
X_{p_i} = X_iI_{p_i}=
\begin{cases}
    X_i & \text{if the $i^{th}$ component survives after the random shock.}\\
    0 & \text{if the $i^{th}$ component does not survive after the random shock.}    
\end{cases} \] 
\begin{definition} (Def 2.4, \citet{li2015ordering}) Let $\phi\,: [0,\,\infty) \to [0,\,1]$ be a non-increasing and continuous function such that $\phi(0)=1$ and $\phi(+\infty)=0$, let $\psi=\phi^{-1}$ be the pseudo-inverse of $\phi$, the function $C_{\phi}$ defined as $$ C_{\phi}(x_1,\,x_2,\cdots,\,x_n)=\phi(\psi(x_1)+\psi(x_2)+\cdots+\psi(x_n))\,\,\,\mbox{for}\,\,\, (x_1,\,x_2,\cdots,x_n)\in[0,\,1]^n$$ is called Archimedean copula and $\phi$ is known as Archimedean generator for $C_{\phi}$ if $(-1)^k\phi^{(-k)}\geq0$ for $k=0,1,...,n-2$ and $(-1)^{(n-2)}\phi^{(n-2)}$ is non-increasing and convex. Archimedean copula gives us joint CDF in terms of the generator when the marginal CDFs are given. Similarly, Archimedean survival copula $\hat{C}_{\phi}$ is defined. It relates the survival function of joint distribution to the survival functions of marginal distributions.\newline\newline
\textbf{Note:}
 If $X = (X_1, \ldots, X_n)$ is a random vector with marginal CDFs $F_1,\cdots, F_n$ and the Archimedean copula generator $\phi$, then the CDF of $n$-th order statistics $X_{n:n}$ is,
$$
\begin{aligned} 
 F_{X_{n:n}}(x)&=\mathbb{P}(X_1\le x,\cdots, X_n\le x)\\&=C_{\phi}(F_1(x),\, F_2(x),\cdots,\, F_n(x)) \\
 &=\phi(\psi(F_1(x))+\psi(F_2(x))+\cdots+\psi(F_n(x))).
 \end{aligned}
 $$
 Now if $\phi$ is the generator of the Archimedean survival copula, Then the survival function of the $1$-st order statistics $X_{1:n}$ is, 
$$
\begin{aligned} 
 \overline{F}_{X_{1:n}}(x)&=\mathbb{P}(X_1> x,\cdots, X_n> x)\\ &=\hat{C}_{\phi}(\bar{F}_1(x),\,\bar{F}_2(x),\cdots,\,\bar{F}_n(x)) \\
 &=\phi(\psi(\bar{F}_1(x))+\psi(\bar{F}_2(x))+\cdots+\psi(\bar{F}_n(x))). 
 \end{aligned}
 $$
\end{definition}
\begin{definition} Consider two random variables X and Y with CDFs F and G, survival functions $\overline{F}$ and $\overline{G}$, respectively. Then Y is called greater than X in the sense of usual stochastic order ( i.e., $X \preceq_{st} Y$) if $\overline{F}(x) \leq \overline{G}(x)$ for all $x \in \mathbb{R}^+$.
\end{definition}

Now we state some results in the form of lemmas that are useful for the remaining part of the work.
\begin{lemma}  (\citet{marshall2011preservation}, Theorem 5.A.2) Consider the function h, which is increasing, and concave (convex), then $(x_1,\cdots,\,x_n) \preceq^w (\preceq_w) (y_1,\cdots,\,y_n) \implies(h(x_1),\cdots,\,h(x_n)) \preceq^w  (\preceq_w) (h(y_1),\cdots,\,h(y_n)$ for $\bm{x},\,\bm{y} \in\mathbb{R}^n$.
\end{lemma}
\begin{lemma} (\citet{marshall2011preservation}, Theorem 3.A.4) If a function $\Phi: \mathbb{R}^n \rightarrow \mathbb{R}$ is continuously differentiable and permutation-symmetric, then it is Schur-convex (or Schur-concave) on $\mathbb{R}^n$ if and only if $\Phi$ is symmetric on $\mathbb{R}^n$ and satisfies, for all $i \neq j$,
\[
(x_i - x_j) \bigg(\frac{\partial \Phi(x)}{\partial x_i} - \frac{\partial \Phi(x)}{\partial x_j}\bigg) \geq (\leq) 0 \quad \text{for all } x \in \mathbb{R}^n.
\]
\end{lemma}

\begin{lemma} (\citet{marshall2011preservation}, Def A.1., ch 3) Consider the function $f: C \subseteq \mathbb{R}^n \to \mathbb{R}$ then, $x\stackrel{m}{\preceq}y \implies f(x) \leq (\geq) f(y)$ if and only if $f$ is a Schur-convex (Schur-concave) function on $\mathbb{R}^n$.
\end{lemma}
\begin{lemma} (\citet{marshall2011preservation}, Theorem 3.A.8) Consider the function $f : C \subseteq \mathbb{R}^n \rightarrow \mathbb{R}$ then, $x \preceq^w (\preceq_w) y \implies f(x) \leq f(y)$ if and only if $f$ is decreasing (increasing) and Schur-convex function on $\mathbb{R}^n$.
\end{lemma}
\begin{lemma} (\citet{marshall2011preservation}, Theorem A.8) Consider the function $f: C \subseteq \mathbb{R}^n \rightarrow \mathbb{R}$ then, for $x, y \in C$,  $x \preceq_w y$ $\implies  f(x) \leq(\geq) f(y)$ if and only if $f$ is both increasing  (decreasing) and Schur-convex (Schur-concave) function on $C\subseteq \mathbb{R}^n$.
\end{lemma}

\begin{lemma}(\citet{li2015ordering}, Lemma A.1 and \citet{nelsen2006archimedean}, Theorem 4.4.2)
Let $C_1$ and $C_2$ be two $n$-dimentional Archimedean copulas with generators $\phi_1$ and $\phi_2$ (with pseudo-inverses $\psi_1$ and $\psi_2$) respectively, if $\psi_2 \circ \phi_1$ is a super-additive (sub-additive) function, then $C_1(\bf{u})\leq(\geq) C_2(\bf{u})$ for all $\bf{u} =(u_1,u_2, \cdots, u_n) \in [0,1]^n$. That is, $$ \phi_1\bigg(\sum_{i=1}^{n}\psi_1(u_i)\bigg) \leq(\geq) \phi_2\bigg(\sum_{i=1}^{n}\psi_2(u_i)\bigg). $$
\end{lemma}
\section{Main results for the super-additive case}
We state the first result in the form of a proposition which generalizes Theorem 3.4 of \citet{amini2024comparison} in the case when systems possess distinct dependency structures. It compares the reliability of two parallel systems when components' survival functions are increasing and log-concave in parameters of the distributions, components are subject to random shocks, and the composition function $\psi_2 \circ \phi_1$ is a super-additive function.
\begin{proposition} 
 Consider two dependent and heterogenous random vectors $X^p = (X_1I_{p_1}, \ldots, X_nI_{p_n})$ and $Y^p = (Y_1I_{p_1}, \ldots, Y_nI_{p_n}) $, where $X_i\sim F(x, \alpha_i)$, $Y_i\sim F(x, \beta_i)$ and $I_{p_i}\sim Bernoulli(p_i)$ are independent of $X_i$'s and $Y_i$'s for $i\in \{1,\cdots, n\}$. Further, consider that $X^p$ and $Y^p$ possess Archimedean copulas generated by $\phi_1$ and $\phi_2$ respectively and $\textbf{p} = (p_1, \ldots, p_n)$. Then $Y_{n:n}^p \preceq_{st} X_{n:n}^p$ if 

 $(i)$ $\bm{\alpha} \preceq^m \bm{\beta}$, $\psi_2\circ \phi_1$ is super-additive and $\bm{\alpha},\, \bm{\beta},\, \bm{p} \in \mathbb{D^+}.$
 
 $(ii)$ $ t\psi_1^{\prime}(1-t)$ or $ t\psi_2^{\prime}(1-t)$ is increasing in \(t\).

 $(iii)$ $\overline{F}(x,\theta)$ be increasing and log-concave function of the parameter $\theta$.
 
\end{proposition}

\begin{proof} The CDF of the random variable $Y_{n:n}^p$ is 
\[F_{Y_{n:n}^p}(x) = \phi_2\left(\sum_{j=1}^{n} \psi_2\left(1 - p_j \overline{F}(x, \beta_j)\right)\right).\] 
Using super-additivity of $\psi_2\circ \phi_1$ and Lemma 2.6 (originally Lemma A.1. of \citet{li2015ordering}), we get that \begin{equation}\phi_2\left(\sum_{j=1}^{n} \psi_2\left(1 - p_j \overline{F}(x, \alpha_j)\right)\right) \geq \phi_1\left(\sum_{j=1}^{n}   \psi_1\left(1 - p_j \overline{F}(x, \alpha_j)\right)\right). \label{3.1.1} \end{equation}
Now using Theorem 3.4 from \citet{amini2024comparison} along with assuming $\phi_2$ as the generator of underlying Archimedean copulas, we can deduce that $F_{Y_{n;n}^p, \beta}(x)\ge F_{Y_{n;n}^p, \alpha}(x) $ that is,  \begin{equation}\phi_2\left(\sum_{j=1}^{n} \psi_2\left(1 - p_j \overline{F}(x, \beta_j)\right)\right) \geq \phi_2\left(\sum_{j=1}^{n} \psi_2\left(1 - p_j \overline{F}(x, \alpha_j)\right)\right).\label{3.1.2} \end{equation} 
Combining \eqref{3.1.1} and \eqref{3.1.2} we get $\phi_2\left(\sum_{j=1}^{n} \psi_2\left(1 - p_j \overline{F}(x, \beta_j)\right)\right) \ge \phi_1\left(\sum_{j=1}^{n} \psi_1\left(1 - p_j \overline{F}(x, \alpha_j)\right)\right)$. That is $F_{Y_{n:n}^p}(x)\ge F_{X_{n:n}^p}(x)$.
\end{proof} 

\begin{remark} Proposition 3.1. removes the limitation that both the parallel systems possess the same dependency structures in Theorem 3.4 of \citet{amini2024comparison}, as it is quite natural that the systems possess distinct dependency structures. It provides insight into the fact that when the survival functions of the components are increasing and log-concave, systems possess different dependency structures, and components are experiencing random shocks with a certain probability of surviving, we can establish a stochastic ordering of the parallel system's lifetimes. It is important to note that there are many copula generators and survival functions satisfying conditions in Proposition 3.1. A few examples are provided below.\newline\newline
(i) Regarding conditions (i) and (ii) of Proposition 3.1, consider $\phi_1$ and $\phi_2$ be two generators of Clayton copula defined by $\phi_i(x)=(1+\theta_i x)^{\frac{1}{\theta_i}}$ and inverses $\psi_i(x)=\frac{1}{\theta_i}(t^{-\theta_i}-1)$ with parameters $\theta_1$ and $\theta_2$ respectively, then for $\theta_1, \theta_2 \in (0,\, \infty)$ we have, $$(\psi_2 \circ \phi_1)^{\prime\prime}(x)=(\theta_2-\theta_1)(1+\theta_1 x)^{\frac{\theta_2}{\theta_1}-2}$$ this implies $\psi_2 \circ \phi_1$  is a convex function and hence a super-additive function for $\theta_2\geq \theta_1>0$ and $x\psi_i^{\prime}(1-x)=\frac{x}{(1-x)^{\theta_i +1}}$ for $i\in \{1,\,2\}$ are increasing functions of $x\in (0,\,1)$. For another example, consider two AMH copulas with parameters $\theta_1,\,\theta_2 \in [0\,\,1)$ and $0\leq \theta_1 < \theta_2$, then both the condition (i) and (ii) are satisfied [see \citet{zhang2019stochastic} Remark 3.5 for reference].\newline \newline
(ii) Regarding condition (iii) on the survival functions of the components, we can find many survival functions that are increasing and log-concave simultaneously in parameters. For example, exponential distribution with mean $\lambda$ and survival function $\bar{F}(x,\,\lambda)=e^{-\frac{x}{\lambda}}$, log-logistic distribution with $\bar{F}(x,\,\alpha,\,\beta)=\frac{\alpha^\beta}{\alpha^\beta+x^\beta}$ (increasing and log-concave w.r.t scale parameter $\alpha>0$), parallel systems when components are from proportional reversed hazard model, etc.
\end{remark}
\begin{remark}
  In particular when, $\phi_1=\phi_2$, then $\psi_2 \circ \phi_1(x)=x$ is a super-additive function and then above proposition 3.1. becomes the Theorem 3.4 of \citet{amini2024comparison}. Therefore Theorem 3.4 of \citet{amini2024comparison} is a particular case of Proposition 3.1. when the dependency structure among the components of two parallel systems $X^p$ and $Y^p$ are described by the same Archimedean copulas with generator $\phi_1=\phi_2$. \newline \newline
(i) It can be proved that the majorization condition $\alpha \preceq^m \beta$ in Proposition 3.1. can be more generalized to the super-majorization $\alpha \preceq^w \beta$ with the conclusion of the proposition being preserved. therefore Proposition 3.1. is true in more general parallel systems where parameter vectors are showing super-majorization order.
\end{remark} 
\begin{example} Consider random vectors $X^p =(X_1I_{p_1}, X_2I_{p_2}, X_3I_{p_3})$ and $Y^p = (Y_1I_{p_1}, X_2I_{p_2}, Y_3I_{p_3})$ where $X_i \sim F(x,\,\alpha_i)=1-e^{-\frac{x}{\alpha_i}}$ and $Y_i \sim F(x,\,\beta_i)=1-e^{-\frac{x}{\beta_i}}$ where, $\bm{\alpha}=(10,\,3,\,1)$, $\bm{\beta}=(11,\,2,\,1)$ and
$\bm{p}=(0.8,\, 0.3,\, 0.2) \in \mathbb{D^+}$. Then it is easy to observe that $\bar{F}(x,\,\alpha_i)$ is increasing in $\alpha_i>0$ and $\log \bar{F}(x,\,\alpha_i)=-\frac{x}{\alpha_i}$ is a concave function of $\alpha_i$.\newline 

Now consider the multivariate AMH copulas generated by $\phi_1(x)=\frac{1-\theta_1}{e^x-\theta_1}$ and $\phi_2(x)=\frac{1-\theta_2}{e^x-\theta_2}$ so that $\psi_2(x)=\log\left(\frac{1-\theta_2+\theta_2 x}{x}\right)$ then we get, $x\psi_2^{\prime}(1-x)=\frac{x(1-\theta_2)}{(1-x)(1-\theta_2 x)}$ is an increasing function of $x\in(0,\,1)$ and $\psi_2 \circ \phi_1(x)= \log\left(\frac{1-\theta_2}{1-\theta_1}(e^x-\theta_1)+\theta_2\right) $. Some straightforward manipulations gives $(\psi_2 \circ \phi_1)^{\prime}(x)=\left(\frac{\theta_2-\theta_1}{1-\theta_2}e^{-x}+1\right)^{-1}$ which is an increasing function for $x\in [0,\,1]$ and  $0\leq \theta_1<\theta_2<1$, implies it is a convex function with $\psi_2 \circ \phi_1(0)=0$. Hence $\psi_2 \circ \phi_1$ is a super-additive function, see Remark 3.5 of \citet{zhang2019stochastic} for reference. 

The graph of $F_{Y_{n:n}^p}(x)-F_{X_{n:n}^p}(x)$ is shown in \textbf{Fig.1} for $\theta_1=0.2$ and $ \theta_2=0.75$ to illustrate it is a non-negative function of $x>0$.

\begin{figure}[ht]
    \centering
    \includegraphics[width=0.7\linewidth]{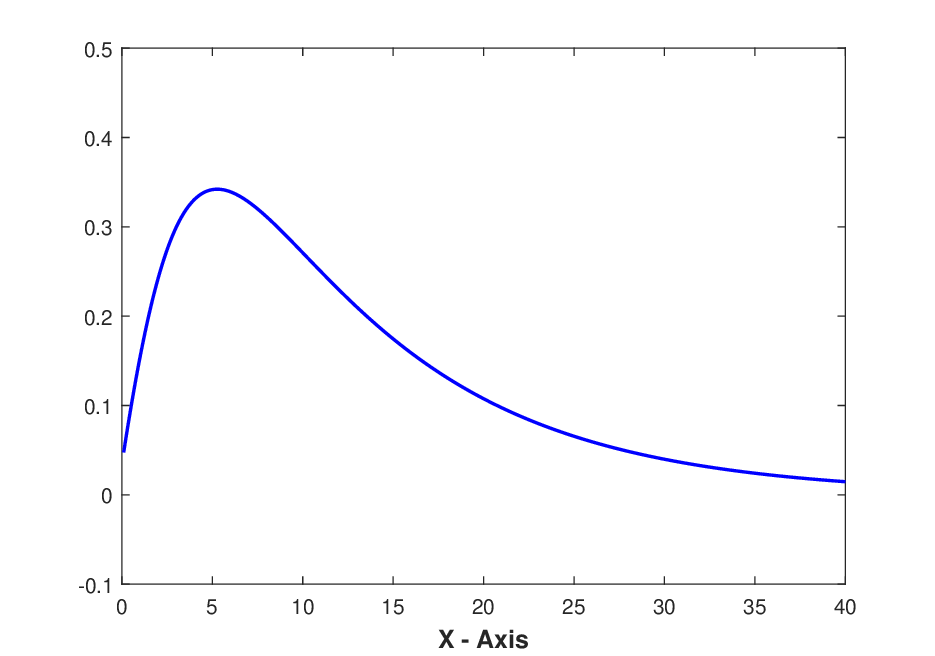}
    \caption{Graph of $F_{Y_{n:n}^p}(x)-F_{X_{n:n}^p}(x)$ for $\theta_1=0.2$ and $ \theta_2=0.75$.}
    \label{fig:1}
\end{figure}
\end{example}
Now a natural question arises that if in Proposition 3.1., the components' survival functions are decreasing and log-convex instead of increasing and log-concave in parameters, besides, $ t\psi_1^{\prime}(1-t)$ or $ t\psi_2^{\prime}(1-t)$ is decreasing instead of increasing in \(t\ge 0\), then can we have the stochastic ordering between two parallel systems having different dependency structures? The answer to this question is affirmative. The theorem presented below is regarding the usual stochastic ordering between two parallel systems when the survival functions  $\overline{F}(x,\theta)$ of the components are decreasing and the log-convex function of the parameter $\theta$, $ t\psi_1^{\prime}(1-t)$ or $ t\psi_2^{\prime}(1-t)$ is decreasing in \(t\), components may experience random shock and parallel systems may show different dependency structures with $\psi_2 \circ \phi_1$ is a super-additive function.
\begin{theorem}
 Consider two dependent and heterogeneous random vectors $X^p = (X_1I_{p_1}, \ldots, X_nI_{p_n})$ and $Y^p = (Y_1I_{p_1}, \ldots, Y_nI_{p_n}) $, where $X_i\sim F(x, \alpha_i)$, $Y_i\sim F(x, \beta_i)$ and $I_{p_i}\sim Bernoulli(p_i)$ are independent of $X_i$'s and $Y_i$'s for $i\in \{1,\cdots, n\}$. Further, consider that $X^p$ and $Y^p$ possess Archimedean copulas generated by $\phi_1$ and $\phi_2$ respectively and $\textbf{p} = (p_1, \ldots, p_n)$.  Then $Y_{n:n}^p \preceq_{st} X_{n:n}^p$ if 

 $(i)$ $ \bm{\beta} \stackrel{m}{\preceq} \bm{\alpha}$, $\psi_2\circ \phi_1$ is super-additive and $\bm{\alpha},\, \bm{\beta} \in \mathbb{D^+},\,\bm{p} \in\mathbb{I^+}. $
 
 $(ii)$ $ t\psi_1^{\prime}(1-t)$ or $ t\psi_2^{\prime}(1-t)$ is decreasing in \(t\).

 $(iii)$ $\overline{F}(x,\theta)$ be decreasing and log-convex function of the parameter $\theta$.
 
\end{theorem}
\begin{proof} The CDF of the random variable $Y_{n:n}^p$ is 
\[F_{Y_{n:n}^p}(x) = \phi_2\left(\sum_{j=1}^{n} \psi_2\left(1 - p_j \overline{F}(x, \beta_j)\right)\right).\]

Using super-additivity of $\psi_2\circ \phi_1$ and Lemma 2.6 we get that \begin{equation}\phi_2\left(\sum_{j=1}^{n} \psi_2\left(1 - p_j \overline{F}(x, \alpha_j)\right)\right) \geq \phi_1\left(\sum_{j=1}^{n}   \psi_1\left(1 - p_j \overline{F}(x, \alpha_j)\right)\right). \label{3.2.1} \end{equation}
  Now we will prove that   \begin{equation}\phi_2\left(\sum_{j=1}^{n} \psi_2\left(1 - p_j \overline{F}(x, \beta_j)\right)\right) \geq \phi_2\left(\sum_{j=1}^{n} \psi_2\left(1 - p_j \overline{F}(x, \alpha_j)\right)\right).\label{3.2.2} \end{equation} 
  For this, we will consider the function $L$ defined as  
 $ L(\bm{\beta}) = \phi_2\left(\sum_{j=1}^{n} \psi_2\left(1 - p_j \overline{F}(x, \beta_j)\right)\right)$.\newline
 Now we have,  $$ \frac{\partial L(\bm{\beta})}{\partial \beta_j}=-\phi_2^{\prime}\left(\sum_{j=1}^{n} \psi_2\left(1 - p_j \overline{F}(x, \beta_j)\right)\right)\cdot \psi_2^{\prime}\left(1 - p_j \overline{F}(x, \beta_j)\right) \cdot p_j \frac{\partial \overline{F}(x, \beta_j)}{\partial \beta_j}.$$
We will shows that $L$ is a Schur-concave function. Now consider,
\newline

$\frac{\partial L(\bm{\beta})}{\partial \beta_i} - \frac{\partial L(\bm{\beta})}{\partial \beta_j}=$\newline
$-\phi_2^{\prime}\left(\sum_{j=1}^{n} \psi_2\left(1 - p_j \overline{F}(x, \beta_j)\right)\right)\bigg( \psi_2^{\prime}\left(1 - p_i \overline{F}(x, \beta_i)\right) \cdot p_i \frac{\partial \overline{F}(x, \beta_i)}{\partial \beta_i} - \psi_2^{\prime}\left(1 - p_j \overline{F}(x, \beta_j)\right) \cdot p_j \frac{\partial \overline{F}(x, \beta_j)}{\partial \beta_j}\bigg)$. \newline

$\stackrel{sign}{=}\bigg( \psi_2^{\prime}\left(1 - p_i \overline{F}(x, \beta_i)\right) \cdot p_i\overline{F}(x, \beta_i) \frac{\frac{\partial \overline{F}(x, \beta_i)}{\partial \beta_i}}{\overline{F}(x, \beta_i)} - \psi_2^{\prime}\left(1 - p_j \overline{F}(x, \beta_j)\right) \cdot p_j \overline{F}(x, \beta_j)\frac{\frac{\partial \overline{F}(x, \beta_j)}{\partial \beta_j}}{\overline{F}(x, \beta_j)}\bigg)$. \newline

Without loss of generality suppose that $1\le i\le j\le n$. Note that $p_i\le p_j$ and $\beta_i\geq \beta_j$ since $\bm{p} \in \mathbb{I^+},\,\,\bm{\beta} \in \mathbb{D^+}$. Now using  $p_i \overline{F}(x, \beta_i) \le p_j\overline{F}(x, \beta_j)$, decreasing property of $ t\psi_2^{\prime}(1-t)$ and log-convex property of $\overline{F}(x, \theta)$ in $\theta$, we get, \newline 
$(\beta_i-\beta_j)\bigg(\frac{\partial L(\bm{\beta})}{\partial \beta_i} - \frac{\partial L(\bm{\beta})}{\partial \beta_j}\bigg)$ $$\stackrel{sign}{=}(\beta_i-\beta_j)
\bigg( \psi_2^{\prime}\left(1 - p_i \overline{F}(x, \beta_i)\right) \cdot p_i\overline{F}(x, \beta_i) \frac{\frac{\partial \overline{F}(x, \beta_i)}{\partial \beta_i}}{\overline{F}(x, \beta_i)} - \psi_2^{\prime}\left(1 - p_j \overline{F}(x, \beta_j)\right) \cdot p_j \overline{F}(x, \beta_j)\frac{\frac{\partial \overline{F}(x, \beta_j)}{\partial \beta_j}}{\overline{F}(x, \beta_j)}\bigg)\le 0.$$
$Lemma\,2.2.$ implies that $L$ is Schur-concave function. Further, using the Schur-concave property of $L$ and $Lemma\, 2.3.$, $\bm{\beta} \stackrel{m}{\preceq} \bm{\alpha}\implies L(\bm{\alpha}) \le L(\bm{\beta})$, which is inequality \eqref{3.2.2}. Combining \eqref{3.2.1} and \eqref{3.2.2} we get $\phi_2\left(\sum_{j=1}^{n} \psi_2\left(1 - p_j \overline{F}(x, \beta_j)\right)\right) \ge \phi_1\left(\sum_{j=1}^{n} \psi_1\left(1 - p_j \overline{F}(x, \alpha_j)\right)\right)$. That is $F_{Y_{n:n}^p}(x)\ge F_{X_{n:n}^p}(x)$.
\end{proof}
\begin{remark} As the conditions in our Theorem 3.2. and in Proposition 3.1. has a significant difference, it enables us to make stochastic ordering between two parallel systems in different situations than those specified in Proposition 3.1. Moreover, Theorem 3.2 provides insight into the fact that when the survival functions of the components are decreasing and log-convex, systems possess different dependency structures with $\psi_2 \circ \phi_1$ is a super-additive function and components are experiencing random shocks, we can establish a stochastic ordering of the parallel systems' lifetimes. Some examples of Archimedean generators satisfying the conditions of $Theorem \, 3.2.$ are as follows.\newline

(i) consider the Gumbel copula with Archimedean generators defined by $\phi_i(x)=e^{-x^{\frac{1}{\theta_i}}}$ where $\theta_i \in[1,\, \infty)$ for $ i=1,2$. Then $\psi_i(x) = (-\log(x))^{\theta_i}$ and $\psi_2 \circ \phi_1(x)= x^{\frac{\theta_2}{\theta_1}}$ which is a super-additive function if $1\leq \theta_1 \leq \theta_2$. Further, $x\psi_i^{\prime}(1-x)=\frac{\theta_ix(-\log(1-x))^{\theta_i-1}}{x-1}$ is a non-increasing function of $x\in (0,\,1)$ for $\theta_i \in [1,\,\infty)$.\newline

(ii) Consider Gumbel-barnett copulas with Archimedean generators $\phi_i(x)=e^{\frac{1}{\theta_i}(1-e^x)}$ where $\theta_i \in (0,\,1]$ for $i=1,2$. Then $\psi_i(x)=\log(1-\theta_i\log(x))$ and $\psi_2 \circ \phi_1(x)=\log(1-\frac{\theta_2}{\theta_1}(1-e^x))$. Differentiating two times we have, $$(\psi_2 \circ \phi_1)^{\prime\prime}(x)=\frac{e^{-x}(1-\frac{\theta_2}{\theta_1})}{(\frac{\theta_2}{\theta_1}(e^{-x}-1)-e^{-x})^2}\geq0,\quad if \quad 0<\theta_2\leq\theta_1<1.$$ Therefore $\psi_2 \circ \phi_1$ is a convex function, hence a super-additive function for $0<\theta_2\leq\theta_1<1$. Further, $x\psi_i^{\prime}(1-x)=\frac{x\theta_i}{(1-x)(\theta_i\log(1-x)-1)}$ is decreasing function for $x\in(0,\,1)$, where $\theta_i \in (0,\,1]$.\newline

(iii) It's worth mentioning that survival functions  $\overline{F}(x,\,\alpha_i)$ of the components is decreasing and log-convex function in parameters $\alpha_i$ in Theorem 3.2., is very common in practice and such examples of survival functions can be found easily. Some examples of such survival functions are Weibull distribution (in scale parameter $\alpha>0$) with survival functions $\bar{F}(x,\,\alpha,\,\beta)=e^{-\alpha x^{\beta}}$, exponential distribution with mean $\frac{1}{\lambda}$, components following proportional hazard model, additive hazard model etc.\newline \end{remark}

 \begin{example} Condsider random vectors $X^p =(X_1I_{p_1}, X_2I_{p_2}, X_3I_{p_3})$ and $Y^p = (Y_1I_{p_1}, X_2I_{p_2}, Y_3I_{p_3})$ where $X_i \sim F(x,\,\alpha_i)=1-e^{-\alpha_i x^2}$ and $Y_i \sim F(x,\,\beta_i)=1-e^{-\beta_i x^2}$ (Weibull distribution with shape parameter $k=2$) where, $\bm{\alpha}=(8,\,5,\,1)$, $\bm{\beta}=(7,\,4,\,3) \in \mathbb{D^+}$ and $\bm{p}=(0.1,\, 0.4,\, 0.9) \in \mathbb{I^+}$. Then observe that $\bar{F}(x,\,\alpha_i)$ is decreasing in $\alpha_i>0$ and $\log \bar{F}(x,\,\alpha_i)=-\alpha_i x^2$ is a convex function of $\alpha_i$.\newline  

Now consider the multivariate Gumbel-Barnett copulas generated by  $\phi_i(x)=e^{\frac{1}{\theta_i}(1-e^x)}$ where $\theta_i \in (0,\,1]$ for $i=1,2$ in the $Remark\,\,3.3.(ii)$. Then all the conditions of  $Theorem\,3.2.$ are satisfied for $0<\theta_2\leq\theta_1<1$.\newline

The graph of $F_{Y_{n:n}^p}(x)-F_{X_{n:n}^p}(x)$ is shown in \textbf{Fig.2} for $\theta_1=0.6$ and $ \theta_2=0.3$ to illustrate it is a non-negative function of $x>0$. 

\begin{figure}[ht]
    \centering
    \includegraphics[width=0.7\linewidth]{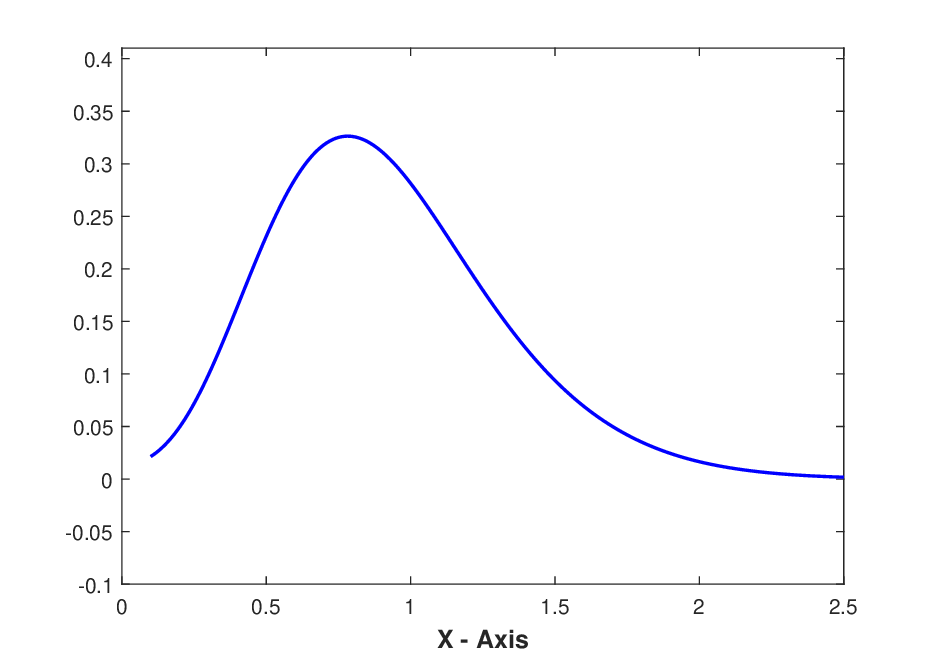}
    \caption{Graph of $F_{Y_{n:n}^p}(x)-F_{X_{n:n}^p}(x)$ for $\theta_1=0.6$ and $\theta_2=0.3$.}
    \label{fig:2}
\end{figure}
\end{example}

 The following theorem is regarding the usual stochastic ordering of lifetimes of two series systems when components' survival functions $\overline{F}(x,\,\alpha_i)$ are increasing and log-concave in parameters $\alpha_i(>0)$ for $i\in\{ 1,\,,2,\cdots,\,n\}$, the pseudo-inverse functions $\psi_1$ or $\psi_2$ satisfy that $t\psi_1^{\prime}(t)$\, or \,$t\psi_2^{\prime}(t)$ is increasing function, systems show different dependency structures and components are subject to random shocks.
\begin{theorem}
 Consider two dependent and heterogenous random vectors $X^p = (X_1I_{p_1}, \ldots, X_nI_{p_n})$ and $Y^p = (Y_1I_{p_1}, \ldots, Y_nI_{p_n}) $, where $X_i\sim F(x, \alpha_i)$, $Y_i\sim F(x, \beta_i)$ and $I_{p_i}\sim Bernoulli(p_i)$ are independent of $X_i$'s and $Y_i$'s for $i\in \{1,\cdots, n\}$. Further, consider that $X^p$ and $Y^p$ possess Archimedean survival copulas generated by $\phi_1$ and $\phi_2$ respectively and $\textbf{p} = (p_1, \ldots, p_n)$.  Then $X_{1:n}^p \preceq_{st} Y_{1:n}^p$ if 

 $(i)$ $\bm{\beta} \stackrel{m}{\preceq} \bm{\alpha}$, $\psi_2\circ \phi_1$ is super-additive and $\bm{\alpha}$, $\bm{\beta}$, $\bm{p} \in \mathbb{D^+}.$ 
 
 $(ii)$ $t\psi_1^{\prime}(t)$\, or \,$t\psi_2^{\prime}(t)$ is increasing function.  

$(iii)$ $\overline{F}(x,\theta)$ be increasing and log-concave function of the parameter $\theta$.
\end{theorem} 

\begin{proof} The survival function of $Y_{1:n}^p$ is $\overline{F}_{Y_{1:n}^p}(x) = \phi_2\left(\sum_{j=1}^{n} \psi_2\left( p_j \overline{F}(x,\,\beta_j)\right)\right)$.

Using super-additivity of $\psi_2\circ \phi_1$ and Lemma 2.6 we get that \begin{equation}\phi_2\left(\sum_{j=1}^{n} \psi_2\left( p_j \overline{F}(x,\,\alpha_j)\right)\right) \geq \phi_1\left(\sum_{j=1}^{n} \psi_1\left( p_j \overline{F}(x,\,\alpha_j)\right)\right). \label{3.3.1} \end{equation} 
  Now we will prove that   \begin{equation}\phi_2\left(\sum_{j=1}^{n} \psi_2\left( p_j \overline{F}(x,\,\beta_j)\right)\right) \geq \phi_2\left(\sum_{j=1}^{n} \psi_2\left( p_j \overline{F}(x,\,\alpha_j)\right)\right)\label{3.3.2} \end{equation} 
  For this, we will consider the function $L$ defined as  
 $ L(\bm{\beta}) =\phi_2\left(\sum_{j=1}^{n} \psi_2\left(p_j \overline{F}(x,\,\beta_j)\right)\right).$ \newline
We now show that it is a Schur-concave  function.\newline
Here, \begin{equation} \frac{\partial L(\bm{\beta})}{\partial \beta_j}=\phi_2^{\prime}\left(\sum_{j=1}^{n} \psi_2\left( p_j \overline{F}(x,\,\beta_j)\right)\right)\cdot \psi_2^{\prime}\left( p_j \overline{F}(x,\,\beta_j)\right) \cdot  p_j \frac{\partial \overline{F}(x,\,\beta_j)}{\partial \beta_j}. \label{3.3.3} \end{equation} 
 Now consider, 
\newline
$\frac{\partial L(\bm{\beta})}{\partial \beta_i} - \frac{\partial L(\bm{\beta})}{\partial \beta_j}=
\phi_2^{\prime}\left(\sum_{j=1}^{n} \psi_2\left( p_j \overline{F}(x,\,\beta_j)\right)\right)\cdot$ $$\bigg(p_i \overline{F}(x,\,\beta_i)\psi_2^{\prime}\left( p_i \overline{F}(x,\,\beta_i)\right) \cdot \frac{\frac{\partial \overline{F}(x,\,\beta_i)}{\partial \beta_i}}{\overline{F}(x,\,\beta_i)} - p_j \overline{F}(x,\,\beta_j) \cdot \psi_2^{\prime}\left( p_j \overline{F}(x,\,\beta_j)\right) \frac{\frac{\partial \overline{F}(x,\,\beta_j)}{\partial \beta_j}}{\overline{F}(x,\,\beta_j)}\bigg).$$
Without loss of generality suppose that $1\le i\le j\le n$. Note that $p_i\geq p_j$ and $\beta_i\geq\beta_j$ since $\bm{p},\, \bm{\beta} \in \mathbb{D^+}$. Now using $t_i=p_i \overline{F}(x,\,\beta_i) \ge t_j=p_j \overline{F}(x,\,\beta_j)$, increasing property of $t\psi_2^{\prime}(t)$, non-positivity of $\phi_2^{\prime}$ and log-cocave property of $\overline{F}(x,\,\beta_j)$ w.r.t $\beta_j>0$. We get, \newline
$(\beta_i-\beta_j)\bigg(\frac{\partial L(\bm{\beta})}{\partial \beta_i} - \frac{\partial L(\bm{\beta})}{\partial \beta_j}\bigg)$ $$\stackrel{sign}{=} (\beta_i-\beta_j) \bigg( p_j \overline{F}(x,\,\beta_j) \cdot \psi_2^{\prime}\left( p_j \overline{F}(x,\,\beta_j)\right) \frac{\frac{\partial \overline{F}(x,\,\beta_j)}{\partial \beta_j}}{\overline{F}(x,\,\beta_j)}-p_i \overline{F}(x,\,\beta_i)\psi_2^{\prime}\left( p_i \overline{F}(x,\,\beta_i)\right) \cdot \frac{\frac{\partial \overline{F}(x,\,\beta_i)}{\partial \beta_i}}{\overline{F}(x,\,\beta_i)}\bigg)\leq 0. $$

$Lemma\,2.2.$ implies that $L$ is Schur-concave function. Further, using the Schur-concave property of $L$ and $Lemma\, 2.3.$, $\bm{\beta} \stackrel{m}{\preceq} \bm{\alpha}\implies L(\bm{\alpha}) \leq L(\bm{\beta})$, which is inequality \eqref{3.3.2}. Combining \eqref{3.3.1} and \eqref{3.3.2} we get $\phi_2\left(\sum_{j=1}^{n} \psi_2\left( p_j \overline{F}(x,\,\beta_j)\right)\right) \ge \phi_1\left(\sum_{j=1}^{n} \psi_1\left( p_j \overline{F}(x,\,\alpha_j)\right)\right)$ that is, $\overline{F}_{Y_{1:n}^p}(x)\ge \overline{F}_{X_{1:n}^p}(x)$.
\end{proof}
 
\begin{remark} The above Theorem 3.3. which compares the reliability of two series systems having different dependency structures under the influence of random shock environment differs from the Theorem 3.8 of \citep{amini2024comparison} when the Archimedean generators $\phi_1$ and $\phi_2$ are not a log-concave function but having the property that $t\psi_1^{\prime}(t)$\, or \,$t\psi_2^{\prime}(t)$ is increasing function, where $\psi_1$ and $\psi_2$ are pseudo-inverse of Archimedean generators $\phi_1$ and $\phi_2$ respectively. It is to be noted that many Archimedean generators satisfy the condition "$t\psi_1^{\prime}(t)$\, or \,$t\psi_2^{\prime}(t)$ is an increasing function" of our Theorem 3.3 but $\phi_1$ and $\phi_2$ are not log-concave functions. For example, consider Gumbel copulas in below Remark 3.5 (i) with Archimedean generators $\phi_1(x)=e^{-x^{\frac{1}{\theta_1}}}$ and $\phi_2(x)=e^{-x^{\frac{1}{\theta_2}}}$ for $\theta_1,\, \theta_2 \in [1,\,\infty)$ then we have, $$ \frac{d^2}{dt^2}\bigg(\log\phi_i(t) \bigg)=\frac{1}{\theta_i}\bigg(1-\frac{1}{\theta_i}\bigg)t^{\frac{1}{\theta_i}-2}\ge0\quad for\quad i=1,2\quad and\quad\theta_i\ge1, \:t\ge0.$$ Which shows that $\phi_1$ and $\phi_2$ are log-convex functions hence not log-concave functions and satisfies "$t\psi_1^{\prime}(t)$\, or \,$t\psi_2^{\prime}(t)$ is an increasing function" as shown in Remark 3.5 (i). Which makes the difference between our Theorem 3.3 and Theorem 3.8 of \citep{amini2024comparison}.\end{remark} 
\begin{remark} Theorem 3.3. provides insight into the fact that when the survival function of the components of series systems is increasing and log-concave, systems possess different dependency structures and components are experiencing random shocks, we can establish a stochastic ordering of the lifetimes of series systems. There are many Archimedean copula generators and survival functions satisfying conditions in Theorem 3.3.\newline
(i) Consider Gumbel copulas with the generator defined by $\phi_1(x)=e^{-x^{\frac{1}{\theta_1}}}$ and $\phi_2(x)=e^{-x^{\frac{1}{\theta_2}}}$ for $\theta_1,\, \theta_2 \in [1,\,\infty)$. Then, It follows that $$\frac{d}{dx}(x\psi_i^{\prime}(x))=\theta_i(\theta_i-1)\frac{(-\log(x))^{\theta_i-2}}{x}$$ are non-negative functions for $i=1,2$ and $\theta_i>2$. Further, $ \psi_2 \circ \phi_1(x)=x^{\frac{\theta_2}{\theta_1}}$ is super-additive for $2<\theta_1<\theta_2$. \newline
(ii) Consider the Clayton copulas  having generator $\phi_1(x)=(1+\theta_1 x)^{-\frac{1}{\theta_1}}$ and $\phi_2(x)=(1+\theta_2 x)^{-\frac{1}{\theta_2}}$. Then simple manipulation gives, $$\frac{d}{dx}\left(x\psi_i^{\prime}(x)\right)=\frac{\theta_i}{x^{\theta_i+1}}\geq0 \,\,and \,\,  \psi_2 \circ \phi_1(x)=\frac{(1+\theta_1 x)^{\frac{\theta_2}{\theta_1}}-1}{\theta_2}$$ is super-additive for $0\leq\theta_1<\theta_2$, are satisfied. For further details, one may refer to page 147 of \citet{amini2024comparison}. On the other hand, examples of increasing and log-concave survival functions are given in $Remark\,\,3.1(ii)$.
\end{remark}
  \begin{example} Consider the random vectors $X^p =(X_1I_{p_1}, X_2I_{p_2}, X_3I_{p_3})$ and $Y^p = (Y_1I_{p_1}, X_2I_{p_2}, Y_3I_{p_3})$ where $X_i \sim \bar{F}(x,\,\alpha_i)=\frac{\alpha_i^2}{\alpha_i^2+x^2}$ and $Y_i \sim \bar{F}(x,\,\beta_i)=\frac{\beta_i^2}{\beta_i^2+x^2}$ where, $\bm{\alpha}=(7,\,4,\,1)$, $\bm{\beta}=(7,\,3,\,2)$ and
$\bm{p}=(0.6,\, 0.4,\, 0.1) \in \mathbb{D^+}$. The survival function of log-logistic distribution, $\bar{F}(x,\,\alpha_i)$ is increasing in $\alpha_i>0$ and $\log \bar{F}(x,\,\alpha_i)=2\log(\alpha_i)-\log(\alpha_i^2+x^2)$ is a concave function of $\alpha_i>0$. \newline
Now consider the Gumbel copulas with the generators defined by $\phi_1(x)=e^{-x^{\frac{1}{\theta_1}}}$ and $\phi_2(x)=e^{-x^{\frac{1}{\theta_2}}}$ for $\theta_1,\, \theta_2 \in [1,\,\infty)$ of $Remark \,\,3.5.(i)$. Then these generators satisfy the required conditions of Theorem 3.3 for $\theta_2 >\theta_1>2$ as shown in $Remark\,\,3.5(i)$.

The graph of $F_{Y_{n:n}^p}(x) - F_{X_{n:n}^p}(x)$ is shown in \textbf{Fig.3} for $\theta_1=3$ and $ \theta_2=15$ to illustrate it is a non-positive function of $x>0$.

\begin{figure}[ht]
    \centering
    \includegraphics[width=0.7\linewidth]{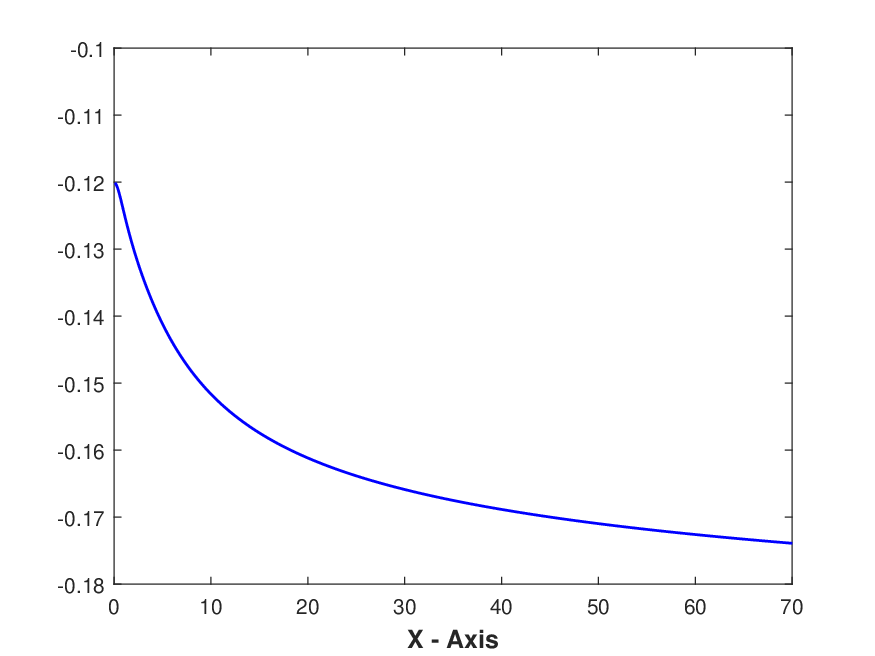}
    \caption{Graph of $F_{Y_{n:n}^p}(x) - F_{X_{n:n}^p}(x)$ for $\theta_1=3$ and $\theta_2=15$.}
    \label{fig:3}
\end{figure}
\end{example}
\newpage
Now a question arises that if the survival functions of the components in a series system are decreasing and log-convex function in parameters instead of increasing and log-concave functions in Theorem 3.3, then can we compare the reliability two systems in the presence of random shock? To find the answer to this question, we present Theorem 3.4 below. The following Theorem 3.4 is regarding the usual stochastic ordering of two series system when the survival functions  $\overline{F}(x,\theta)$ of the components are decreasing and a log-convex function of the parameter $\theta$, components may experience random shock, series systems may show different dependency structures, and $\psi_2\circ \phi_1$ is super-additive function. 
\begin{theorem}
 Consider two dependent and heterogeneous random vectors $X^p = (X_1I_{p_1}, \ldots, X_nI_{p_n})$ and $Y^p = (Y_1I_{p_1}, \ldots, Y_nI_{p_n}) $, where $X_i\sim F(x, \alpha_i)$, $Y_i\sim F(x, \beta_i)$ and $I_{p_i}\sim Bernoulli(p_i)$ are independent of $X_i$'s and $Y_i$'s for $i\in \{1,\cdots, n\}$. Further, consider that $X^p$ and $Y^p$ possess Archimedean copulas generated by $\phi_1$ and $\phi_2$ respectively and $\textbf{p} = (p_1, \ldots, p_n)$.  Then $X_{1:n}^p \leq_{st} Y_{1:n}^p$ if 

 $(i)$ $\alpha \preceq^w \beta$, $\psi_2\circ \phi_1$ is super-additive and $\bm{\alpha}$, $\bm{\beta}\in \mathbb{D^+}$, $\bm{p} \in\mathbb{I^+}$. 
 
 $(ii)$ $t\psi_1^{\prime}(t)$\, or \,$t\psi_2^{\prime}(t)$ is decreasing function.  

$(iii)$ $\overline{F}(x,\theta)$ be decreasing and log-convex function of the parameter $\theta$.
\end{theorem} 

\begin{proof} Proceeding similarly to the proof of $Theorem\,3.3.$ we have, 
\begin{equation} \frac{\partial L(\bm{\beta})}{\partial \beta_j}=\phi_2^{\prime}\left(\sum_{j=1}^{n} \psi_2\left( p_j \overline{F}(x,\,\beta_j)\right)\right)\cdot \psi_2^{\prime}\left( p_j \overline{F}(x,\,\beta_j)\right) \cdot  p_j \frac{\partial \overline{F}(x,\,\beta_j)}{\partial \beta_j} \leq 0. \label{3.2.3} \end{equation}
This shows that $L$ is decreasing in each argument when other arguments are kept constant.\newline
Now consider, 
\newline
$\frac{\partial L(\bm{\beta})}{\partial \beta_i} - \frac{\partial L(\bm{\beta})}{\partial \beta_j}=\phi_2^{\prime}\left(\sum_{j=1}^{n} \psi_2\left( p_j \overline{F}(x,\,\beta_j)\right)\right)\cdot$ $$\bigg(p_i \overline{F}(x,\,\beta_i)\psi_2^{\prime}\left( p_i \overline{F}(x,\,\beta_i)\right) \cdot \frac{\frac{\partial \overline{F}(x,\,\beta_i)}{\partial \beta_i}}{\overline{F}(x,\,\beta_i)} - p_j \overline{F}(x,\,\beta_j) \cdot \psi_2^{\prime}\left( p_j \overline{F}(x,\,\beta_j)\right) \frac{\frac{\partial \overline{F}(x,\,\beta_j)}{\partial \beta_j}}{\overline{F}(x,\,\beta_j)}\bigg).$$
Without loss of generality suppose that $1\le i\le j\le n$. Note that $p_i\leq p_j$ and $\beta_i\geq \beta_j$ since $\bm{p}\in \mathbb{I^+}, \,\bm{\beta}\in \mathbb{D^+}$. Now using $t_i=p_i \overline{F}(x,\,\beta_i) \le t_j=p_j \overline{F}(x,\,\beta_j)$, decreasing property of $t\psi_2^{\prime}(t)$, non-positivity of $\phi_2^{\prime}$ and log-convex property of $\overline{F}(x,\,\beta_j)$ w.r.t $\beta_j>0$. We get, \newline
$(\beta_i-\beta_j)\bigg(\frac{\partial L(\bm{\beta})}{\partial \beta_i} - \frac{\partial L(\bm{\beta})}{\partial \beta_j}\bigg)$ $$\stackrel{sign}{=} (\beta_i-\beta_j) \bigg( p_j \overline{F}(x,\,\beta_j) \cdot \psi_2^{\prime}\left( p_j \overline{F}(x,\,\beta_j)\right) \frac{\frac{\partial \overline{F}(x,\,\beta_j)}{\partial \beta_j}}{\overline{F}(x,\,\beta_j)}-p_i \overline{F}(x,\,\beta_i)\psi_2^{\prime}\left( p_i \overline{F}(x,\,\beta_i)\right) \cdot \frac{\frac{\partial \overline{F}(x,\,\beta_i)}{\partial \beta_i}}{\overline{F}(x,\,\beta_i)}\bigg)\geq 0. $$

$Lemma\,2.2.$ implies that $L$ is Schur-convex function. Further, using the decreasing and Schur-convex property of $L$ and the $Lemma\, 2.4.$, $\alpha \preceq^w \beta\implies L(\bm{\alpha}) \leq L(\bm{\beta})$, which is inequality \eqref{3.3.2}.  Combining \eqref{3.3.1} and \eqref{3.3.2} we get $\phi_2\left(\sum_{j=1}^{n} \psi_2\left( p_j \overline{F}(x,\,\beta_j)\right)\right) \ge \phi_1\left(\sum_{j=1}^{n} \psi_1\left( p_j \overline{F}(x,\,\alpha_j)\right)\right)$ that is, $\overline{F}_{Y_{1:n}^p}(x)\ge \overline{F}_{X_{1:n}^p}(x)$.
\end{proof}

\begin{remark}
Theorem 3.4. sheds light on the fact that when the survival functions of the components are decreasing and log-convex in parameters, systems possess different dependency structures and components are experiencing random shocks, we can establish a stochastic ordering of the lifetimes of series systems.
 \end{remark}
\begin{remark} Series systems with Archimedean generators having properties as required in  Theorem 3.4 are very common. Following examples of Archimedean generators are provided in support of Theorem 3.4.\newline
(i) Regarding conditions of theorem 3.4., consider (Remark 1. of  \citet{ghanbari2023stochastic}) $\phi_1$ and $\phi_2$ be two generators of Gumbel-Haugaard copula given by generators $\phi_1(x)=e^{1-(1+x)^{\theta_1}}$ and $\phi_2(x)=e^{1-(1+x)^{\theta_2}}$, then for $\theta_1, \theta_2 \in [1,\, \infty)$ and $\theta_1\geq\theta_2>1$, we have, $\psi_2 \circ \phi_1(x)=(1+x)^{\frac{\theta_1}{\theta_2}}-1$ is super-additive. Now consider the inverse of the generator  $\psi_i(x)=(1-\log(x))^{\frac{1}{\theta_i}}-1$ for $i \in \{1,\,2\}$ and $x\in(0,\,1)$. Simple manipulation shows that, 
$$\frac{d}{dx}(x\psi_i^{\prime}(x))=\frac{1}{\theta_i}(\frac{1}{\theta_i}-1)\frac{(1-\log(x))^{\frac{1}{\theta_i}-2}}{x}\leq0$$ for $\theta_i \geq1$ and $x\in(0,\,1)$. Therefore the function $x\psi_i^{\prime}(x)$ is decreasing for $x\in(0,\,1)$ and $\theta_i>1$.\newline
 
\end{remark}
\begin{example}  Condsider random vectors $X^p =(X_1I_{p_1}, X_2I_{p_2}, X_3I_{p_3})$ and $Y^p = (Y_1I_{p_1}, X_2I_{p_2}, Y_3I_{p_3})$ where $X_i \sim F(x,\,\alpha_i)=1-e^{-\alpha_i x^3}$ and $Y_i \sim F(x,\,\beta_i)=1-e^{-\beta_i x^3}$ (Weibull distribution with shape parameter $k=3$) where, $\bm{\alpha}=(8,\,6,\,2)$, $\bm{\beta}=(9,\,4,\,2) \in \mathbb{D^+}$ and $\bm{p}=(0.1,\, 0.2,\, 0.7) \in \mathbb{I^+}$. Then  $\bar{F}(x,\,\alpha_i)$ is decreasing in $\alpha_i>0$ and $\log \bar{F}(x,\,\alpha_i)=-\alpha_i x^3$ is a convex function of $\alpha_i$. \newline \newline
Now consider the Gumbel-Hougaard copula in Remark 3.7(i). Then these generators satisfy the required conditions of Theorem 3.4 for $\theta_1 \geq\theta_2>1$ as shown in $Remark\,\,3.7 (i)$.

The graph of $F_{Y_{n:n}^p}(x) - F_{X_{n:n}^p}(x)$ is shown in \textbf{Fig.4} for $\theta_1=9$ and $ \theta_2=2$ to illustrate it is a non-positive function of $x>0$.

\begin{figure}[ht]
    \centering
    \includegraphics[width=0.7\linewidth]{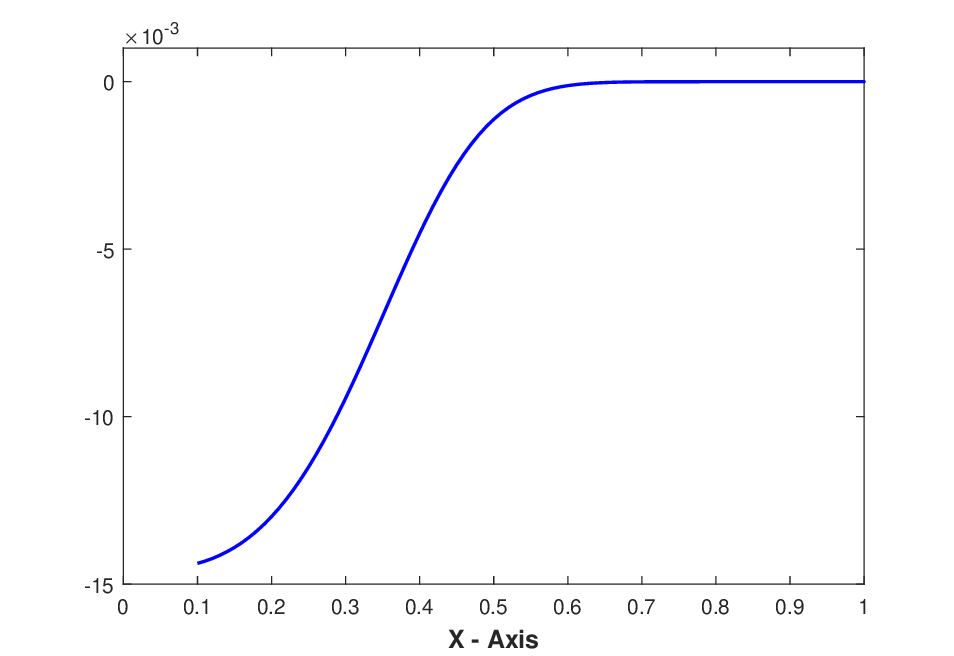}
    \caption{ Graph of $F_{Y_{n:n}^p}(x) - F_{X_{n:n}^p}(x)$ for $\theta_1=9$ and $ \theta_2=2$. }
    \label{fig:4}
\end{figure}
\end{example}
 \newpage
\section{Main results for the sub-additive case} In Section 3 we have discussed the comparison of the reliability of parallel and series systems assuming that the components are either increasing log-concave or decreasing log-convex in parameters in the presence of random shock. We also assume that different dependency structures are shown by different systems and described by Archimedean copulas with generators $\phi_1$ and $\phi_2$. However, we have considered the case when the composition $\psi_2 \circ \phi_1$ is super-additive in all the results in Section 3. But, naturally, there may be a situation where the Archimedean generators satisfy that $\psi_2 \circ \phi_1$ is a sub-additive function rather than a super-additive function. This Section 4 deals with the usual stochastic ordering of first and highest-order statistics (series and parallel systems' lifetime respectively) when the function $\psi_2 \circ \phi_1$ is a sub-additive, the survival functions of components are either an increasing log-concave or a decreasing log-convex in parameters, and systems are experiencing a random shock.

Proposition 4.1 considers the case when the components' survival functions are increasing log-concave in parameters under the presence of random shock and provides sufficient conditions for the usual stochastic ordering of parallel systems possessing different dependency structures, and $\psi_2\circ \phi_1$ is sub-additive function.
 
\begin{proposition} 
 Consider two dependent and heterogenous random vectors $X^p = (X_1I_{p_1}, \ldots, X_nI_{p_n})$ and $Y^p = (Y_1I_{p_1}, \ldots, Y_nI_{p_n}) $, where $X_i\sim F(x, \alpha_i)$ and $Y_i\sim F(x, \beta_i)$ and $I_{p_i}\sim Bernoulli(p_i)$ are independent of $X_i$'s and $Y_i$'s for $i\in \{1,\cdots, n\}$. Further, consider that $X^p$ and $Y^p$ possess Archimedean copulas generated by $\phi_1$ and $\phi_2$ respectively and $\textbf{p} = (p_1, \ldots, p_n)$. Then $X_{n:n}^p \preceq_{st} Y_{n:n}^p$ if 

 $(i)$ $\bm{\beta} \preceq^m \bm{\alpha}$, $\psi_2\circ \phi_1$ is sub-additive and $\bm{\alpha},\, \bm{\beta},\, \bm{p} \in \mathbb{D^+}$.
 
 $(ii)$ $ t\psi_1^{\prime}(1-t)$ or $ t\psi_2^{\prime}(1-t)$ is increasing in \(t\).

 $(iii)$ $\overline{F}(x,\theta)$ be increasing and log-concave function of the parameter $\theta$.
 
\end{proposition}

\begin{proof} The CDF of $Y_{n:n}^p$ is $F_{Y_{n:n}^p}(x) = \phi_2\left(\sum_{j=1}^{n} \psi_2\left(1 - p_j \overline{F}(x, \beta_j)\right)\right).$\newline
Using sub-additivity of $\psi_2\circ \phi_1$ and Lemma 2.6 ( Theorem 4.4.2 of \citet{nelsen2006archimedean})   we get that \begin{equation}\phi_1\left(\sum_{j=1}^{n} \psi_1\left(1 - p_j \overline{F}(x, \alpha_j)\right)\right) \geq \phi_2\left(\sum_{j=1}^{n}   \psi_2\left(1 - p_j \overline{F}(x, \alpha_j)\right)\right). \label{4.1.1} \end{equation}
Now using Theorem 3.4 from \citet{amini2024comparison} along with assuming $\phi_2$ as the generator of underlying Archimedean copulas, we can deduce that $F_{Y_{n;n}^p, \alpha}(x)\ge F_{Y_{n;n}^p, \beta}(x) $ that is,\begin{equation}\phi_2\left(\sum_{j=1}^{n} \psi_2\left(1 - p_j \overline{F}(x, \alpha_j)\right)\right) \geq \phi_2\left(\sum_{j=1}^{n} \psi_2\left(1 - p_j \overline{F}(x, \beta_j)\right)\right).\label{4.1.2} \end{equation} 
 Combining \eqref{4.1.1} and \eqref{4.1.2} we get $F_{X_{n:n}^p}(x)\geq F_{Y_{n:n}^p}(x)$.
\end{proof}
\begin{remark}
  In particular when, $\phi_1=\phi_2$, then $\psi_2 \circ \phi_1(x)=x$ is a sub-additive function and then above proposition 4.1. becomes the Theorem 3.4 of \citet{amini2024comparison}. Therefore Theorem 3.4 of \citet{amini2024comparison} is a particular case of Proposition 4.1. when the dependency structure among the components of two parallel systems $X^p$ and $Y^p$ are described by the same Archimedean copulas with generator $\phi_1=\phi_2$. \newline \newline
(i) It can be proved that the majorization condition $\alpha \preceq^m \beta$ in Proposition 4.1. can be more generalized to the super-majorization $\alpha \preceq^w \beta$ with the conclusion of the proposition being preserved. Therefore Proposition 4.1. is true in more general parallel systems where parameter vectors are showing super-majorization order.\newline\newline
(ii) It is clear that in $Remark\,3.1(i)$, Clayton copula generators give,  $$(\psi_2 \circ \phi_1)^{\prime\prime}(x)=(\theta_2-\theta_1)(1+\theta_1 x)^{\frac{\theta_2}{\theta_1}-2}$$ this implies $\psi_2 \circ \phi_1$  is a concave function and hence a sub-additive function for $\theta_1\geq \theta_2>0$ and $x\psi_i^{\prime}(1-x)=\frac{x}{(1-x)^{\theta_i +1}}$ for $i\in \{1,\,2\}$ are increasing functions of $x\in (0,\,1)$. Similarly for the AMH copula family if $0<\theta_2\leq\theta_1<1$ then $\psi_2 \circ \phi_1$ becomes a concave function, therefore, a sub-additive function.
\end{remark}

Theorem 4.2 considers the case when the survival functions of the components are decreasing and the log-convex function of parameters and it specifies the sufficient conditions for the usual stochastic ordering of two parallel systems assuming that the components may experience random shock and parallel systems may show different dependency structures with $\psi_2\circ \phi_1$ is a sub-additive function. 
\begin{theorem}
 Consider two dependent and heterogeneous random vectors $X^p = (X_1I_{p_1}, \ldots, X_nI_{p_n})$ and $Y^p = (Y_1I_{p_1}, \ldots, Y_nI_{p_n}) $, where $X_i\sim F(x, \alpha_i)$ and $Y_i\sim F(x, \beta_i)$ and $I_{p_i}\sim Bernoulli(p_i)$ are independent of $X_i$'s and $Y_i$'s for $i\in \{1,\cdots, n\}$. Further, consider that $X^p$ and $Y^p$ possess Archimedean copulas generated by $\phi_1$ and $\phi_2$ respectively and $\textbf{p} = (p_1, \ldots, p_n)$. Then $X_{n:n}^p \preceq_{st} Y_{n:n}^p$ if 

 $(i)$ $ \bm{\alpha} \stackrel{m}{\preceq} \bm{\beta}$, $\psi_2\circ \phi_1$ is sub-additive and $\bm{\alpha},\, \bm{\beta} \in \mathbb{D^+},\,\bm{p} \in\mathbb{I^+}. $
 
 $(ii)$ $ t\psi_1^{\prime}(1-t)$ or $ t\psi_2^{\prime}(1-t)$ is decreasing in \(t\).

 $(iii)$ $\overline{F}(x,\theta)$ be decreasing and log-convex function of the parameter $\theta$.
 
\end{theorem}
\begin{proof} The CDF of $Y_{n:n}^p$ is $F_{Y_{n:n}^p}(x) = \phi_2\left(\sum_{j=1}^{n} \psi_2\left(1 - p_j \overline{F}(x, \beta_j)\right)\right).$\newline
Using sub-additivity of $\psi_2\circ \phi_1$ and Lemma 2.6 ( Theorem 4.4.2 of \citet{nelsen2006archimedean})   we get that, \begin{equation}\phi_1\left(\sum_{j=1}^{n} \psi_1\left(1 - p_j \overline{F}(x, \alpha_j)\right)\right) \geq \phi_2\left(\sum_{j=1}^{n}   \psi_2\left(1 - p_j \overline{F}(x, \alpha_j)\right)\right). \label{4.2.1} \end{equation}
  Now we will show that  \begin{equation}\phi_2\left(\sum_{j=1}^{n} \psi_2\left(1 - p_j \overline{F}(x, \alpha_j)\right)\right) \geq \phi_2\left(\sum_{j=1}^{n} \psi_2\left(1 - p_j \overline{F}(x, \beta_j)\right)\right)\label{4.2.2} \end{equation} 
  For this, we will consider the function $L$ defined as  
 $ L(\bm{\beta}) = \phi_2\left(\sum_{j=1}^{n} \psi_2\left(1 - p_j \overline{F}(x, \beta_j)\right)\right)$.\newline
After that, following similar steps to Theorem 3.2 gives, $L$ is a Schur-concave function. Using Lemma 2.3, $\bm{\alpha} \stackrel{m}{\preceq} \bm{\beta} \implies L(\bm{\beta})\leq L(\bm{\alpha})$. Which is the inequality \eqref{4.2.2}. Combining \eqref{4.2.1} and \eqref{4.2.2} we get $F_{X_{n:n}^p}(x)\geq F_{Y_{n:n}^p}(x)$. 
\end{proof} 
\begin{remark} Consider the Gumbel copula in $Remark\,\,3.3(i)$, we have, $$(\psi_2\circ \phi_1)^{\prime\prime}(x)=\frac{\theta_2}{\theta_1}(\frac{\theta_2}{\theta_1}-1)x^{\frac{\theta_2}{\theta_1}-2}.$$Now if $1<\theta_2\leq\theta_1$ then $\psi_2\circ \phi_1$ becomes a concave function hence a sub-additive function. For other conditions see $Remark\,\,3.3(i)$. Similarly, if we consider the Gumbel-barnett copula in $Remark\,3.3(ii)$ we have, $$(\psi_2 \circ \phi_1)^{\prime\prime}(x)=\frac{e^{-x}(1-\frac{\theta_2}{\theta_1})}{(\frac{\theta_2}{\theta_1}(e^{-x}-1)-e^{-x})^2}.$$ Now if $1<\theta_1\leq\theta_2$ then $\psi_2\circ \phi_1$ becomes a concave function hence a sub-additive function.
\end{remark}
The following Theorem 4.2 is concerned about the usual stochastic ordering of two series system's lifetimes and provides sufficient conditions for the case when components' survival functions $\overline{F}(x,\,\alpha_i)$ are increasing and a log-concave in parameters $\alpha_i(>0)$ for $i\in\{ 1,\,,2,\cdots,\,n\}$, systems show different dependency structures, and components are subject to random shock with $\psi_2\circ \phi_1$ is sub-additive function.
\begin{theorem}
 Consider two dependent and heterogenous random vectors $X^p = (X_1I_{p_1}, \ldots, X_nI_{p_n})$ and $Y^p = (Y_1I_{p_1}, \ldots, Y_nI_{p_n}) $, where $X_i\sim F(x, \alpha_i)$, $Y_i\sim F(x, \beta_i)$ and $I_{p_i}\sim Bernoulli(p_i)$ are independent of $X_i$'s and $Y_i$'s for $i\in \{1,\cdots, n\}$. Further, consider that $X^p$ and $Y^p$ possess Archimedean survival copulas generated by $\phi_1$ and $\phi_2$ respectively and $\textbf{p} = (p_1, \ldots, p_n)$. Then $Y_{1:n}^p \preceq_{st} X_{1:n}^p$ if 

 $(i)$ $\bm{\alpha} \stackrel{m}{\preceq} \bm{\beta}$, $\psi_2\circ \phi_1$ is sub-additive and $\bm{\alpha}$, $\bm{\beta}$, $\bm{p} \in \mathbb{D^+}.$ 
 
 $(ii)$ $t\psi_1^{\prime}(t)$\, or \,$t\psi_2^{\prime}(t)$ is increasing function.  

$(iii)$ $\overline{F}(x,\theta)$ be increasing and log-concave function of the parameter $\theta$.
\end{theorem} 

\begin{proof} The survival function of the random variable $Y_{1:n}^p$ is $\overline{F}_{Y_{1:n}^p}(x) = \phi_2\left(\sum_{j=1}^{n} \psi_2\left( p_j \overline{F}(x,\,\beta_j)\right)\right)$.

Using sub-additivity of $\psi_2\circ \phi_1$ we get that, \begin{equation}\phi_1\left(\sum_{j=1}^{n} \psi_1\left( p_j \overline{F}(x,\,\alpha_j)\right)\right) \geq \phi_2\left(\sum_{j=1}^{n} \psi_2\left( p_j \overline{F}(x,\,\alpha_j)\right)\right). \label{4.3.1} \end{equation} 
  Now we will show that,   \begin{equation}\phi_2\left(\sum_{j=1}^{n} \psi_2\left( p_j \overline{F}(x,\,\alpha_j)\right)\right) \geq \phi_2\left(\sum_{j=1}^{n} \psi_2\left( p_j \overline{F}(x,\,\beta_j)\right)\right).\label{4.3.2} \end{equation} 
  For this, we will consider the function $L$ defined as  
 $ L(\bm{\beta}) =\phi_2\left(\sum_{j=1}^{n} \psi_2\left(p_j \overline{F}(x,\,\beta_j)\right)\right).$ \newline
After that, following similar steps to Theorem 3.3., we get $L$ is a Schur-concave function. Using Lemma 2.3, $\bm{\alpha} \stackrel{m}{\preceq} \bm{\beta} \implies L(\bm{\beta}) \leq L(\bm{\alpha})$. Which is the inequality \eqref{4.3.2}. Combining \eqref{4.3.1} and \eqref{4.3.2} we get $\bar{F}_{X_{1:n}^p}(x)\geq \bar{F}_{Y_{1:n}^p}(x)$. 
 \end{proof}
\begin{remark} Consider Gumbel copula in $Remark\,\,3.5(i)$, where generators are defined by $\phi_1(x)=e^{-x^{\frac{1}{\theta_1}}}$ and $\phi_2(x)=e^{-x^{\frac{1}{\theta_2}}}$ for $\theta_1,\, \theta_2 \in [1,\,\infty)$. Then we have, $$\frac{d}{dx}(x\psi_i^{\prime}(x))=\theta_i(\theta_i-1)\frac{(-\log(x))^{\theta_i-2}}{x}$$ are non-negative functions for $i=1,2$, $\theta_i>2$ and $x\in(0,\,1)$. Further, $ \psi_2 \circ \phi_1(x)=x^{\frac{\theta_2}{\theta_1}}$ is sub-additive for $2<\theta_2<\theta_1$. Similarly, for Clayton copula generators in $Remark\,\,3.5(ii)$, $$\frac{d}{dx}\left(x\psi_i^{\prime}(x)\right)=\frac{\theta_i}{x^{\theta_i+1}}\geq0 \,\,and \,\,  \psi_2 \circ \phi_1(x)=\frac{(1+\theta_1 x)^{\frac{\theta_2}{\theta_1}}-1}{\theta_2}$$ is sub-additive for $0\leq\theta_2<\theta_1$ and $x>0$. 
 \end{remark}
 The following Theorem 4.4 concerns the usual stochastic ordering of series systems when the survival functions  $\overline{F}(x,\theta)$ of the components of series systems are decreasing and log-convex function of the parameter $\theta$, components may experience random shock and series systems may show different dependency structures with $\psi_2\circ \phi_1$ is sub-additive function.

\begin{theorem}
 Consider two dependent and heterogeneous random vectors $X^p = (X_1I_{p_1}, \ldots, X_nI_{p_n})$ and $Y^p = (Y_1I_{p_1}, \ldots, Y_nI_{p_n}) $, where $X_i\sim F(x, \alpha_i)$, $Y_i\sim F(x, \beta_i)$ and $I_{p_i}\sim Bernoulli(p_i)$ are independent of $X_i$'s and $Y_i$'s for $i\in \{1,\cdots, n\}$. Further, consider that $X^p$ and $Y^p$ possess Archimedean copulas generated by $\phi_1$ and $\phi_2$ respectively and $\textbf{p} = (p_1, \ldots, p_n)$. Then $Y_{1:n}^p \leq_{st} X_{1:n}^p$ if 

 $(i)$ $\bm{\beta} \preceq^w \bm{\alpha}$, $\psi_2\circ \phi_1$ is sub-additive and $\bm{\alpha}$, $\bm{\beta}\in \mathbb{D^+}$, $\bm{p} \in\mathbb{I^+}$. 
 
 $(ii)$ $t\psi_1^{\prime}(t)$\, or \,$t\psi_2^{\prime}(t)$ is decreasing function.  

$(iii)$ $\overline{F}(x,\theta)$ be decreasing and log-convex function of the parameter $\theta$.
\end{theorem} 

\begin{proof} Proceeding in the similar manner to the proof of $Theorem\,4.3.$ we have,  \begin{equation}\phi_1\left(\sum_{j=1}^{n} \psi_1\left( p_j \overline{F}(x,\,\alpha_j)\right)\right) \geq \phi_2\left(\sum_{j=1}^{n} \psi_2\left( p_j \overline{F}(x,\,\alpha_j)\right)\right). \label{4.4.1} \end{equation} 
  Now we will show that   \begin{equation}\phi_2\left(\sum_{j=1}^{n} \psi_2\left( p_j \overline{F}(x,\,\alpha_j)\right)\right) \geq \phi_2\left(\sum_{j=1}^{n} \psi_2\left( p_j \overline{F}(x,\,\beta_j)\right)\right).\label{4.4.2} \end{equation} 
  For this, we will consider the function $L$ defined as  
 $ L(\bm{\beta}) =\phi_2\left(\sum_{j=1}^{n} \psi_2\left(p_j \overline{F}(x,\,\beta_j)\right)\right).$ \newline
 After that, following similar steps as Theorem 3.4., we get $L$ is decreasing and Schur-convex function. Using Lemma 2.4, $\bm{\beta} \preceq^w \bm{\alpha} \implies L(\bm{\beta})\leq L(\bm{\alpha})$. Which is the inequality \eqref{4.4.2}. Combining \eqref{4.4.1} and \eqref{4.4.2} we get $\bar{F}_{X_{1:n}^p}(x)\geq \bar{F}_{Y_{1:n}^p}(x)$. \end{proof}
\begin{remark} Consider Gumbel-Hougaard copula in $Remark\,\,3.7(i)$, given by the generators, \newline$\phi_1(x)=e^{1-(1+x)^{\theta_1}}$ and $\phi_2(x)=e^{1-(1+x)^{\theta_2}}$, then we have, $$(\psi_2 \circ \phi_1)^{\prime\prime}(x)=\frac{\theta_1}{\theta_2}(\frac{\theta_1}{\theta_2}-1)(1+x)^{\frac{\theta_1}{\theta_2}-2}$$ is a non-positive function for $\theta_1, \theta_2 \in [1,\, \infty)$, $\theta_2\geq\theta_1>1$, and $x>0$. Therefore $\psi_2 \circ \phi_1$ is a concave function, hence a sub-additive function for $\theta_2\geq\theta_1>1$. Now consider the inverse of the generator  $\psi_i(x)=(1-\log(x))^{\frac{1}{\theta_i}}-1$ for $i \in \{1,\,2\}$ and $x\in(0,\,1)$. A simple manipulation shows that, 
$$\frac{d}{dx}(x\psi_i^{\prime}(x))=\frac{1}{\theta_i}\left(\frac{1}{\theta_i}-1\right)\frac{(1-\log(x))^{\frac{1}{\theta_i}-2}}{x}\leq0$$ for $\theta_i >1$ and $x\in(0,\,1)$. Therefore the function $x\psi_i^{\prime}(x)$ is decreasing for $x\in(0,\,1)$ and $\theta_i>1$.
\end{remark}
\section{Some applications} In reliability theory, systems suffer shocks from external stress factors, stressing the system at random. These random shocks may have non-ignorable effects on the reliability of the system. Now if a component faces random shocks, it may survive or fail to do so. If $X_i$ is the lifetime of a component before considering random shock factor, then $X_iI_{p_i}$, where $I_{p_i}$ is Bernoulli random variable with mean $p_i$, is the lifetime of the same component considering the random shock factor presence in the system. Here $p_i$ is the probability that the $i^{th}$ component survives after receiving a random shock. Also, the failure of one component in the system may significantly affect the failure of the other components. Therefore the components' lifetime under the presence of random shock factor in the system are interdependent, one may refer to \citet{article1}. In this article, our theorems consider the components' lifetime distribution having either increasing log-concave or decreasing log-convex survival functions in parameters, which is abundant in statistical analysis. Our theorem provides sufficient conditions for stochastic comparison of series (parallel) systems' lifetime and helps to find the most reliable series (parallel) system among many, under given assumptions.\newline

In actuarial science, suppose $X_i$ is the claim size of $i^{th}$ insurance policyholder, and is a random amount. Now a policyholder may or may not claim for the policy. If the policyholder claims the insurance amount $X_i$ with probability $p_i$ then its random claim amount can be represented by the random variable $X_iI_{p_i}$. Now insurance claims by different policyholders are not always independent, for example, considering the case of a car accident then there is a high probability that the car insurance claim and medical claim of that car driver can occur simultaneously. Therefore different kinds of insurance claims are interdependent, one may refer to \citet{balakrishnan2018ordering}. Now if the company wants to compare the reliability of minimum $X_{1:n}^p$ (or maximum $X_{n:n}^p$) insurance claims size among different policies to identify the most reliable policy and determine the premium amount of that policy, then our theorems can be applied if the random claim sizes having survival functions either increasing log-concave or decreasing log-convex in parameters.\newline

In auction theory, bidders (sellers) present their sealed bids to the auctioneer (buyer) who then solicits the purchase of goods or services at the start of the auction. The highest bidder wins the auction and receives payment from the auctioneer in the amount of the highest price. Government agencies and big corporations frequently use this kind of auction to buy goods and services. If there are n bidders submitting prices $X_1$, ···, $X_n$, then the auctioneer's cost is $X_{n:n}$. Unexpected events may cause some buyers to withdraw from the auction before it starts. Therefore $I_{p_i}X_i$ is the auction price of $i^{th}$ bidder considering that he may or may not withdraw from bidding. Consequently, the last auction cost is the lowest order statistics derived from $I_{p_1}X_1$, ···, and $I_{p_n}X_n$, where $I_{p_i}$ indicates whether $i^{th}$ bidder withdraw or not. Here submitting prices by different bidders are interdependent and their dependency is described by Archimedean copulas, one may refer to \citet{article1}. Our theorems help to compare the reliability of the highest (lowest) bidder prices under the assumption that bidder prices have distributions with either increasing log-concave or decreasing log-convex survival functions of parameters. 
\section{Conclusion}
In this article, we studied the usual stochastic ordering of series and parallel systems when the components' survival functions are either an increasing log-concave or a decreasing log-convex function in parameters, and components may experience random shocks. We also consider that the systems exhibit different dependency structures. An interesting observation is that, if $\phi_1$(with inverse $\psi_1$) and $\phi_2$(with inverse $\psi_2$) are Archimedean generators for two random vectors of concerning series (or parallel) systems respectively, having lifetimes $X^p$ and $Y^p$ respectively, then the usual stochastic order $X^p\preceq_{st}Y^p$ reversed to $Y^p\preceq_{st}X^p$ when the composition $\psi_2 \circ\phi_1$ is changed from super-additive to sub-additive. Further investigation could focus on the comparison of the reliability of the fail-safe system (second-order statistics) under the analogous situations on the component level as well as the system level.
\section*{Funding} The author, Sarikul Islam, is financially supported by the Department of Science and Technology (DST) of the Ministry of Science and Technology, India, through a research grant.
\section*{Acknowledgments} The author, Sarikul Islam, acknowledges the financial support from DST for this work.
\section*{Declarations} The authors have no competing interests to declare that are relevant to the content of this article.
\section*{Conflict of interest} On behalf of all authors, the corresponding author states that there is no conflict of interest.
\bibliographystyle{apalike} 
\bibliography{ref} 
\end{document}